\documentclass[a4paper]{amsart}
\usepackage{amssymb}
\usepackage{amsthm}
\usepackage{amsmath,amscd}
\usepackage[mathscr]{euscript}
\usepackage[all]{xy}
\usepackage{color}
\usepackage[utf8]{inputenc}

\usepackage[scaled=0.90]{helvet} 
\usepackage{courier} 
\normalfont
\usepackage[T1]{fontenc}

\usepackage[T1]{fontenc}
\usepackage[colorlinks=true,linkcolor=red,citecolor=blue]{hyperref}
\newtheorem{theo}{Theorem}[section]
\newtheorem{lemm}[theo]{Lemma}
\newtheorem{prop}[theo]{Proposition}
\newtheorem{coro}[theo]{Corollary}

\theoremstyle{definition}
\newtheorem{defi}[theo]{Definition}
\newtheorem{cons}[theo]{Construction}

\newtheorem{rem}[theo]{Remark}

\newtheorem*{theo*}{Theorem}

\numberwithin{equation}{section}

\newcommand{\opp}{^{\mathrm{op}}}
\newcommand{\cat}{\mathrm}
\newcommand{\icat}{\mathbf}
\newcommand{\oper}{\mathscr}
\newcommand{\on}{\operatorname}
\newcommand{\id}{\mathrm{id}}
\newcommand{\R}{\mathbb{R}}
\renewcommand{\L}{\mathbb{L}}
\newcommand{\Map}{\on{Map}}
\newcommand{\map}{\on{map}}
\newcommand{\Hom}{\mathrm{Hom}}
\newcommand{\End}{\mathrm{End}}
\newcommand{\Aut}{\mathrm{Aut}}
\newcommand{\Fun}{\on{Fun}}

\newcommand{\Mat}{\on{Mat}}
\newcommand{\h}{\widehat}
\newcommand{\Pro}{\on{Pro}}

\newcommand{\CC}{\on{C}^*}
\newcommand{\C}{\on{C}_*}

\newcommand{\QQ}{\mathbb{Q}}

\newcommand{\kk}{\Bbbk}
\newcommand{\ZZ}{\mathbb{Z}}
\renewcommand{\SS}{\mathbb{S}}

\newcommand{\sslash}{\mathbin{/\mkern-6mu/}}
\newcommand{\pab}{\oper{P}a\oper{B}}
\newcommand{\pacd}{\oper{P}a\oper{CD}}
\newcommand{\od}{\oper{D}}
\newcommand{\pugt}{\widehat{\underline{\mathrm{GT}}}}

\newcommand{\pgt}{\widehat{\mathrm{GT}}}
\newcommand{\gt}{\mathrm{GT}}

\newcommand{\op}{\cat{Op}}
\newcommand{\wop}{\cat{WOp}}

\newcommand{\calg}[1]{\cat{Alg}^{E_{\infty}}_{#1}}

\newcommand{\pG}{\h{\cat{G}}}
\renewcommand{\S}{\cat{S}}
\newcommand{\pS}{\h{\cat{S}}}

\newcommand{\G}{\cat{G}}

\newcommand{\icalg}[1]{\icat{CAlg}_{#1}}
\newcommand{\is}{\icat{S}}
\newcommand{\ips}{\h{\icat{S}}}

\newcommand{\iop}{\icat{Op}}
\newcommand{\imod}{\icat{Mod}}

\renewcommand{\sp}{\icat{Sp}}
\newcommand{\psp}{\h{\icat{Sp}}}

\title{Automorphisms of the little disks operad with torsion coefficients}

\author{Geoffroy Horel}

\begin{document}

\address{LAGA (UMR 7539), Département de Mathématiques, Université Paris 13, 99 avenue Jean-Baptiste Clément , 93430 Villetaneuse, France}
\email{horel@math.univ-paris13.fr}


\begin{abstract}
We compute the automorphisms of the Bousfield-Kan completion at a prime $p$ of the little two-disks operads and show that they are given by the pro-$p$ Grothendieck-Teichmüller group. We also show that the Grothendieck-Teichmüller group acts faithfully on the $p$-complete stable little disks operad.
\end{abstract}

\maketitle

\tableofcontents

\section{Introduction}

This paper is concerned with the homotopy automorphisms of the operad $\od$ of little $2$-disks with $p$-torsion coefficients, both unstably and stably. In \cite{horelprofinite}, we studied the automorphisms of the profinite completion of $\od$. The profinite completion being a slightly unfamiliar construction for non-nilpotent spaces (such as the spaces that constitute the little $2$-disks operad), we have decided to write this paper in order to explain the consequences of the main result of \cite{horelprofinite} in terms of constructions that homotopy theorists may be more used to (e.g. the Bousfield-Kan completion and the theory of Hopf cooperads). We also prove results for the automorphisms of the stable version of $\od$  that parallel the known results in the rational case due to Tamarkin and Willwacher in \cite{tamarkinaction,willwacherkontsevich}.

In \cite{drinfeldquasi}, Drinfel'd  constructs the Grothendieck-Teichm\"uller group $\pgt_p$. This is a profinite group that receives a map $\on{Gal}(\bar{\QQ}/\QQ)$. The group $\pgt_p$ is also the group of units of a monoid $\pugt_p$ which is a certain explicit submonoid of the monoid of endomorphism of $\h{(F_2)}_p$, the pro-$p$ completion of the free group on $2$ generators. There is a map $\chi:\pgt_p\to\ZZ_p^\times$ called the cyclotomic character that fits in a commutative diagram
\[\xymatrix{
\on{Gal}(\bar{\QQ}/\QQ)\ar[d]\ar[r]&\h{\ZZ}^\times\cong\prod_{p}\ZZ_p^\times\ar[d]\\
\pgt_p\ar[r]_{\chi}&\ZZ_p^\times
}
\]
where the top map is the cyclotomic character of $\on{Gal}(\bar{\QQ}/\QQ)$ and the right hand side map is the projection on the $p$-th factor.

Given an operad $\oper{P}$ in simplicial sets, and a field $\kk$, we can apply the singular cochains functor levelwise. The resulting object denoted $\CC(\oper{P},\kk)$ is an operad in the opposite of the category of dg-$E_\infty$-algebras over $\kk$ (we call such a structure a Hopf cooperad). Our first main theorem is the following.

\begin{theo}\label{theo: main theorem hopf}
Let $\kk$ be an algebraically closed field of characteristic $p$, then there is an isomorphism of groups
\[\pgt_p\cong\Aut_{\on{Ho}\icat{HCOp}_\kk}(\CC(\od,\kk)),\]
where the automorphism group on the right is computed in the homotopy category of Hopf cooperads over $\kk$.
\end{theo}

This should be compared with the following analogous theorem due to Benoit Fresse in the rational case. 

\begin{theo}[Fresse, \cite{fressehomotopy}]
Let $\kk$ be a field of characteristic zero. Then there is a model for $\CC(\od,\kk)$ that is a commutative  Hopf dg-cooperad (i.e. a cooperad in commutative dg-algebras over $\kk$). Moreover there is an isomorphism of groups
\[\on{GT}(\kk)\cong \Aut_{\on{Ho}\icat{HCOp}_\kk}(\CC(\od,\kk)),\]
where the automorphism group on the right is computed in the homotopy category of dg-commutative Hopf cooperads over $\kk$ and $\on{GT}(\kk)$ denotes the group of $\kk$ points of the pro-algebraic Grothendieck-Teichm\"uller group.
\end{theo}

Note that in the rational case, any $E_\infty$-algebra can be strictified to a commutative dg-algebra but this is not true over fields of positive characteristic. This is the reason why our theorem is stated for a less rigid notion of Hopf cooperad than Fresse's theorem.

Our main result also has a version for the Bousfield-Kan $p$-completion with respect to $\ZZ/p$. Since the Bousfield-Kan completion only preserves products up to weak equivalences, the levelwise application of completion to an operad only has the structure of a weak operad (i.e. a homotopy coherent algebra over the algebraic theory that controls the structure of operads).

\begin{theo}\label{theo: main theorem BK}
Let $(\ZZ/p)_\infty$ denotes the Bousfield-Kan completion with respect to the ring $\ZZ/p$ (as defined in \cite[Chapter I]{bousfieldhomotopy}). Then there is an isomorphism of groups
\[
\pgt_p\cong\Aut_{\on{Ho}\icat{WOpS}}((\ZZ/p)_{\infty}\od),
\]
where the automorphism group is computed in the homotopy category of weak operads in spaces.
\end{theo}

Let us mention that, in the category of spaces, we have a rigidification results giving us an equivalence between the homotopy theory of weak operads in spaces and the homotopy theory of strict operads in spaces (this is the main result of \cite{bergnerrigidification}).

Both of these theorems actually follow from the more fundamental Theorem \ref{theo : main p} about the automorphism of the pro-$p$-completion of the operad of little $2$-disks. This pro-$p$-completion is an operad in the $\infty$-category of pro-$p$-spaces from which one can recover the Hopf cooperad $\CC(\od,\kk)$ and the weak operad $(\ZZ/p)_\infty\od$. 

We also study the stable automorphisms of the little $2$-disks operads with $p$-torsion coefficients. In that case, the rational case is understood by the following theorm due to Willwacher.

\begin{theo}[Willwacher, \cite{willwacherkontsevich}]
There is an isomorphism
\[\on{GT}(\QQ)\cong \Aut_{\on{Ho}\iop\imod_{H\QQ}}(H\QQ\wedge\Sigma^{\infty}_+\od),\]
where the automorphism group on the right is computed in the homotopy category of operads in $H\QQ$-modules
\end{theo}

There are two options for what the $p$-torsion analogue of this theorem should be. The most obvious thing would be to try to study the automorphisms of $H\ZZ/p\wedge \Sigma_+^\infty\od$ or equivalently of the operad in chain complexes $\C(\od,\ZZ/p)$. There is an alternative which is in our sense more natural. Since rationalization is a smashing localization, the category of $H\QQ$-module is exactly the category of rational spectra. Thus the above theorem can equivalently be stated in the following way.

\begin{theo}
Let $L_\QQ\Sigma^{\infty}_+\od$ be the localization of the stabilization of the operad $\od$ at rational homology. Then there is an isomorphism
\[\on{GT}(\QQ)\cong \Aut_{\on{Ho}\iop\sp}(L_{\QQ}\Sigma^{\infty}_+\od),\]
where the automorphism group on the right is computed in the homotopy category of operads in spectra.
\end{theo}

In the present paper, we study the automorphisms of the operad in spectra $L_p\Sigma^{\infty}_+\od$ where $L_p$ denotes localization with respect to ordinary homology with mod-$p$ coefficients. This localization also coincides with the localization at the mod-$p$ Moore spectrum for connective spectra. The object $L_p\Sigma^{\infty}_+\od$ is a more fundamental object than $H\ZZ/p\wedge\Sigma_+^\infty\od$ in the sense that we can construct the latter from the former but the former remembers more data. For instance the action of the Steendrod algebra on the homology of a spectrum $X$ can be recovered from the knowledge of $L_pX$ but not from the $H\ZZ/p$-module $H\ZZ/p\wedge X$. We are unfortunately unable to compute the full group of automorphisms of the $H\ZZ/p$-localization of $\Sigma^\infty_+\od$ but we prove the following theorem.

\begin{theo}\label{theo: main theorem stable}
There exists an injective morphism of groups
\[\pgt_p\to \Aut_{\on{Ho}\iop\sp_p}(L_p\Sigma^\infty_+\od),\]
where the automorphism group on the right is computed in the homotopy category of operads in $H\ZZ/p$-local spectra.
\end{theo}

This theorem is quite surprising. In the unstable case, it is not hard to prove that the natural action of $\pgt_p$ on the $p$-adic homotopy type of $\od$ is faithful. Indeed, restricting this action to the space of operations of arity $3$, we get an action of $\pgt_p$ on $\h{PB_3}$, the pro-$p$-completion of the pure braid group on $3$ strands. The pro-$p$-completion of the free group on two generators $\h{F_2}$ sits inside $\h{PB_3}$ and the action of $\pgt_p$ restricts to the standard action on $\h{F_2}$,  which is faithful by construction. This is in sharp contrast with the stable case. In that case, we show in Proposition \ref{prop: levelwise triviality} that the action of $\pgt_p$ on each spectrum $L_p\Sigma^\infty_+\od(n)$ factors through the cyclotomic character and in particular is very far from being faithful.

\subsection*{Acknowledgments} I wish to thank the Max Planck Institute for Mathematics and the Hausdorff Institute for Mathematics for their hospitality during the period when this paper was written. This work has benefited greatly from conversations with Benoit Fresse, Thomas Willwacher, Thomas Nikolaus, Tobias Barthel, Drew Heard, Marcy Robertson and Pedro Boavida de Brito.

\section{Pro-$p$-homotopy theory}

We briefly review a few facts about $p$-adic homotopy theory. We denote by $\S$ the category of simplicial sets with the Kan-Quillen model structure. We denote by $\is$ the $\infty$-categorical localization of $\S$, also known as the $\infty$-category of spaces.

Let $X$ be a simplicial set and $\kk$ a commutative ring. We denote by $\CC(X,\kk)$ the cochain complex of singular cochains on $X$ with its differential graded $E_\infty$-algebra structure. This $E_\infty$-structure is constructed explicitly in \cite{bergercombinatorial}. If $X=\{X_i\}_{i\in I}$ is a pro-object in $\S$, we denote by $\CC(X,\kk)$ the $E_\infty$-algebra given by the formula:
\[\CC(X,\kk)=\on{colim}_I\CC(X_i,\kk).\]
Note that since $I$ is a filtered category this colimit can be computed in cochain complexes.

By definition of the pro-category, there is a unique cofiltered limit preserving functor
\[\CC(-,\kk):\Pro(\S)\to(\calg{\kk})\opp\]
such that the composition of this functor with the fully faithful inclusion $\S\to\Pro(\S)$ coincides with the above defined functor $\CC(-,\kk)$. We denote by $\icalg{\kk}$ the $\infty$-category obtained by inverting the quasi-isomorphism in the category of $E_\infty$-algebras over $\kk$. Note that by the main theorem of \cite{hinichrectification}, this coincides with the $\infty$-category of commutative algebras (in the sense of \cite{luriehigheralgebra}) in the symmetric monoidal $\infty$-category of chain complexes over $\kk$ or equivalently in the $\infty$-category of $H\kk$-modules (where $H\kk$ denotes the Eilenberg-MacLane spectrum associated to $\kk$).

We denote by $\is_{p-fin}$ the smallest $\infty$-subcategory of $\is$ containing all the objects $K(\ZZ/p,n)$ and that is stable under finite homotopy limits and retracts. We denote by $\ips_p$ the pro-category of $\is_{p-fin}$. We denote by $X\mapsto \Mat_p(X)$ the unique functor $\ips_{p}\to \is$ that preserves limits and that restricts to the inclusion $\is_{p-fin}\to\is$. The functor $\Mat_p$ has a left adjoint for formal reasons. It is denoted $X\mapsto \h{X}_p$ or just $X\mapsto \h{X}$ when there is no ambiguity and called $p$-completion. Note that this is different (but related) to what is sometimes called $p$-completion in homotopy theory namely the Bousfield localization with respect to the homology theory $\on{H}_*(-,\ZZ/p)$ or the Bousfield-Kan completion at the ring $\ZZ/p$. 

The $\infty$-category $\ips_p$ can be presented as the $\infty$-category underlying the model category $\pS_p$ constructed by Morel in \cite{morelensembles}. The Quillen pair
\begin{equation}\label{eq: quillen pair}
{(-)}:\S\leftrightarrows \pS_p:|-|
\end{equation}
presents the $\infty$-adjunction $(\h{(-)}_p,\Mat_p)$. We refer the reader to the next section for more details about this model structure.

We denote by $\is_{p-fin}^0$ the full subcategory of $\is_{p-fin}$ spanned by the connected spaces. The pro-category $\Pro(\is_{p-fin}^0)$ is a full subcategory of $\ips_p$. One can easily characterize the objects of $\is_{p-fin}^0$ in terms of their homotopy groups.

\begin{prop}\label{prop: characterization of p-finite spaces}
An object $X$ of $\is$ is in $\is_{p-fin}^0$ if and only if it is connected, truncated and all of its homotopy groups are finite $p$-groups.
\end{prop}

\begin{proof}
Assume that $X$ is connected truncated and all of its homotopy groups are finite $p$-groups. Then $X$ is nilpotent by \cite[Proposition 7.3.4.]{barneapro}, and by \cite[Proposition V.6.1]{goersssimplicial}, $X$ can be written as the limite of a finite tower of principal fibrations with structure group $K(\ZZ/p,n)$ and hence is in $\is_{p-fin}^0$. The converse is easy.
\end{proof}

Note that for $\kk$ a field of characteristic $p$, the functor $\CC(-,\kk)$ sends weak equivalences in $\pS_p$ to weak equivalences of $E_\infty$-algebras. Hence we have a well-defined functor $\CC(-,\kk):\ips_p\to (\icalg{\kk})\opp$.

\begin{theo}[Mandell, \cite{mandelleinfinity}]\label{prop: cochains is fully faithful}
For $\kk$ an algebraically closed field of characteristic $p$, the functor 
\[X\mapsto \CC(X,\kk)\]
induces a fully faithful embedding $\Pro(\is_{p-fin}^0)\to(\icalg{\kk})\opp$.
\end{theo}

\begin{proof}
This theorem is due to Mandell although we are using a slightly different language. By definition of the pro-category, it suffices to prove that $\CC(-,\kk)$ is fully faithful on  $\is_{p-fin}^0$ and takes its values in compact objects of $\icalg{\kk}$. Let us denote by $S_\kk(n)$ the free $E_\infty$-algebra on a generator of degree $n$. This is a compact object of the $\infty$-category $\icalg{\kk}$. By \cite[Theorem 6.2.]{mandelleinfinity}, there is a pushout square
\[\xymatrix{
S_\kk(n)\ar[r]^{\id-P_0}\ar[d]&S_\kk(n)\ar[d]\\
\kk\ar[r]&\CC(K(\ZZ/p,n),\kk)
}
\]
in the $\infty$-category $\icalg{\kk}$. This implies that $\CC(K(\ZZ/p,n),\kk)$ is compact.

For a general object $X$ in $\is_{p-fin}$, there exists a tower
\begin{equation}\label{tower}
X=X_n\to X_{n-1}\to\ldots X_0\to\ast
\end{equation}
in which each map $X_k\to X_{k-1}$ is a principal fibration with structure group $K(\ZZ/p,s)$ with $s\geq 1$. At each stage the Eilenberg-Moore spectral sequence converges by \cite{dwyerstrong} this implies that $\CC(X,\kk)$ can be obtained as a finite iterated pushout of compact $E_{\infty}$-algebras. Thus, $\CC(X,\kk)$ is a compact object of $\icalg{\kk}$.

It remains to prove that $\CC(-,\kk)$ is fully faithful when restricted to $\is^0_{p-fin}$. But objects of $\is^0_{p-fin}$ are nilpotent, connected, local with respect to $H_*(-,\ZZ/p)$ and of finite $p$-type, it follows by the main theorem of \cite{mandelleinfinity} that $\CC(-,\kk)$ is fully faithful when restricted to those spaces.
\end{proof}

\section{Model structures on pro-$p$-spaces and pro-$p$-groupoids}

In this section we give more details about the model category $\pS_p$ constructed by Morel that presents the $\infty$-category of pro-$p$-spaces $\ips_p$. We also construct another model category $\pG_p$ of pro-$p$-groupoids that is an essential tool in the proof of Theorem \ref{theo : main p}.

\subsection{Pro-$p$-spaces}

Recall that a profinite set is a topological space which can be expressed as a cofiltered limit of finite sets with the discrete topology. By Stone duality, they can be identified with the compact Hausdorff totally disconnected topological spaces. We denote by $\h{\S}$ the category of simplicial objects in profinite sets. The forgetful functor from profinite sets to sets has a left adjoint that sends a set $S$ to the inverse limit of all finite quotients of $S$. Applying these two functor levelwise, we get an adjunction 
\[\h{(-)}:\S\leftrightarrows \pS:|-|.\]

In \cite{morelensembles}, Morel constructs a model structure denoted $\pS_p$ on $\h{\S}$. The weak equivalences in this model structure are the maps that induce isomorphisms on $\ZZ/p$-cohomology, the cofibrations are the monomorphisms. Moreover, it follows from his proof that this model structure is fibrantly generated with generating fibrations the maps
\[K(\ZZ/p,n)\to \ast \;\;\textrm{and}\;\;L(\ZZ/p,n)\to K(\ZZ/p,n+1)\]
and with generating trivial fibrations the maps
\[L(\ZZ/p,n)\to\ast.\]
We refer the reader to \cite[Section 1.4]{morelensembles} for the relevant definitions. This model structure fits in a Quillen adjunction
\[\widehat{(-)}:\S\leftrightarrows \pS_p:|-|.\]
Together with Barnea and Harpaz, we prove in the last section of \cite{barneapro} that this Quillen adjunction presents the $\infty$-categorical adjunction
\[\widehat{(-)}:\is\leftrightarrows \ips_p:\Mat.\]

\subsection{Pro-$p$-groupoids }

Let $\pG$ be the pro-category of the category of finite groupoids (i.e. the category of groupoids with a finite set of morphisms). A finite group $G$ can be seen as a groupoid with a unique object. We denote by $\ast\sslash G$ that groupoid. More generally for $S$ a finite set with an action of $G$ on the right. we denote by $S\sslash G$ the groupoid whose set of objects $S$ and whose set of morphisms from $s$ to $s'$ is the set of elements $g$ of $G$ such that $s.g=s'$.
 
\begin{prop}\label{prop: model structure on pro-p groupoids}
There is a model structure on $\pG$ whose cofibrations are the maps with the left lifting property against $\ZZ/p\sslash\ZZ/p\to\ast$ and whose trivial cofibrations are the maps with the left lifting property against the maps $\ZZ/p\sslash\ZZ/p\to \ast\sslash\ZZ/p$ and $\ast\sslash\ZZ/p\to\ast$.
\end{prop}

\begin{proof}
The proof that this model structure exists can be done exactly as in \cite[Theorem 4.12]{horelprofinite}. 
\end{proof}

From this proof, we also obtain a characterization of the weak equivalences.

\begin{prop}
The weak equivalences in $\pG_p$ are exactly the maps $f:C\to D$ such that the induced maps $H^0(D,\ZZ/p)\to H^0(C,\ZZ/p)$ and $H^1(D,\ZZ/p)\to H^1(C,\ZZ/p)$ are isomorphisms (see \cite[Definition 4.1 and 4.2]{horelprofinite} for the definition of these cohomology groups). 
\end{prop}

This model structure $\pG_p$ fits in a Quillen adjunction
\[\pi:\pS_p\leftrightarrows\pG_p:B.\]
Here the right adjoint $B$ is the obvious variant of the classifying space functor for profinite groupoids. This fact can easily be deduced from the fact that $B$ sends the generating fibrations and trivial fibrations to generating fibrations and trivial fibrations. Now, we also have the following:

\begin{prop}
The trivial fibrations in $\pG_p$ coincide with the trivial fibrations in $\pG$ with the model structure constructed in \cite[Theorem 4.12]{horelprofinite}.
\end{prop}

\begin{proof}
Since the generating trivial fibrations in $\pG_p$ are trivial fibrations in $\pG$, the trivial fibrations of $\pG_p$ are a subclass of those of $\pG$. In order to prove the converse, it suffices, using \cite[Lemma 4.13]{horelprofinite}, to check that the maps of the form $\on{Codisc}(S)\to\ast$ are trivial fibrations for $S$ a finite set. Any finite set is a retract of the set $(\ZZ/p)^n$ for some $n$. It follows that the map $\on{Codisc}(S)\to\ast$ is a retract of the map $\on{Codisc}((\ZZ/p)^n)\to\ast$. The latter map is isomorphic to the $n$-fold product of the map $\ZZ/p\sslash\ZZ/p\to\ast$ and hence is a trivial fibration in $\pG_p$.
\end{proof}

As a corollary of this proposition, we find that the cofibrations in $\pG_p$ are the same as the cofibrations in $\pG$, that is the maps that are injective on objects. Since there are more weak equivalences in $\pG_p$ that in $\pG$, the model structure $\pG_p$ is a Bousfield localization of $\pG$.

\section{Weak operads}\label{section : weak operads}

There exists a small category $\Psi\opp$ with finite products such that the category of operads in $\cat{Set}$ is equivalent to the category of product preserving functors $\Psi\opp\to\cat{Set}$. To construct this category, we first introduce some notation. We consider infinite sequences $\mathbf{n}=(n_1,\ldots,n_k,\ldots)$ of nonnegative integers, with $n_i=0$ for all but finitely many $i$. To such a sequence, we associate the operad $F(\mathbf{n})$ which corepresents the following functor from operads to sets 
\[P\mapsto \prod_{i=0}^\infty P(i)^{n_i}.\]
We define $\Psi$ to be the full subcategory of $\op\cat{Set}$ spanned by the operads $F(\mathbf{n})$. We observe that in the category $\Psi\opp$, there is an isomorphism
\[F(\mathbf{n})\times F(\mathbf{m})\cong F(\mathbf{n}+\mathbf{m}),\]
where the sum $\mathbf{n}+\mathbf{m}$ is componentwise. It follows that the category $\Psi\opp$ has finite products.

We denote by $N^\Psi$ the fully faithful inclusion functor from the category of operads in $\cat{Set}$ to the category of presheaves over $\Psi$. Observe that the essential image of this functor is precisely the category of functors $\Psi\opp\to\cat{Set}$ that preserve finite products. More generally, for a category $\cat{C}$ with products, we make the following definition:

\begin{defi}
Let $\cat{C}$ be a category with products. An operad in $\cat{C}$ is a product preserving functor $\Psi\opp\to\cat{C}$. 
\end{defi}

We denote by $N^\Psi$ the inclusion functor $\op\cat{C}\to\cat{C}^{\Psi\opp}$ and we call it the operadic nerve. Now assume that $\cat{C}$ is a model category.

\begin{defi}
We say that a diagram $X:\Psi\opp\to\cat{C}$ is a weak operad if it is objectwise fibrant and for any pair of objects $(\psi,\tau)$ in $\Psi$, the map
\[X(\psi\times\tau)\to X(\psi)\times X(\tau)\]
is a weak equivalence. We denote by $\wop\cat{C}$ the category of weak operads in $\cat{C}$.
\end{defi}

Given an $\infty$-category $\icat{C}$ with products, we can define the $\infty$-category $\icat{WOpC}$ of weak operads in $\icat{C}$ as the $\infty$-category of product preserving functors $\Psi\opp\to\icat{C}$.

\begin{prop}\label{prop: weak operads infinity categories}
Let $\cat{C}$ be a combinatorial or cocombinatorial model category and $\icat{C}$ be the underlying $\infty$-category, then the $\infty$-categorical localization $\icat{N}(\wop\cat{C})$ of $\wop\cat{C}$ at the objectwise weak equivalences is equivalent to $\icat{WOpC}$.
\end{prop}

\begin{proof}
Let $\cat{C}^f$ be the full subcategory of $\cat{C}$ spanned by fibrant objects. We will construct an equivalence $\icat{N}(\wop\cat{C})\to \icat{WOp}\icat{N}(\cat{C}^f)$. Note that if $X$ and $Y$ are two fibrant objects of $\cat{C}$, the diagram
\[X\leftarrow X\times Y\to Y\]
induces a limit cone in the $\infty$-category $\icat{N}\cat{C}$. It follows that the obvious map
\[\icat{N}(\wop\cat{C})\to \Fun(\Psi\opp,\icat{N}(\cat{C}^f))\]
factors through $\icat{WOp}\icat{N}(\cat{C}^f)$

Hence, we have a commutative diagram
\[\xymatrix{
\icat{N}(\wop\cat{C})\ar[r]\ar[d]&\icat{WOp}\icat{N}(\cat{C}^f)\ar[d]\\
\icat{N}\Fun(\Psi\opp,\cat{C}^f)\ar[r]\ar[d]&\Fun(\Psi\opp,\icat{N}(\cat{C}^f))\ar[d]\\
\icat{N}\Fun(\Psi\opp,\cat{C})\ar[r]&\Fun(\Psi\opp,\icat{N}(\cat{C}))
}
\]
and we wish to show that the top horizontal map is an equivalence of $\infty$-categories. A standard argument with combinatorial model categories implies that all the maps in the bottom square are equivalences. By definition the map $\icat{WOp}\icat{N}(\cat{C}^f)\to \Fun(\Psi\opp,\icat{N}(\cat{C}^f))$ is fully faithful. The map $\icat{N}(\wop\cat{C}^f)\to \icat{N}\Fun(\Psi\opp,\cat{C}^f)$ is also fully faithful as can be seen easily using Dwyer-Kan's hammock localization model for the $\infty$-categorical localization. It follows that the top horizontal map is fully faithful. 

On the other hand, let $F$ be an object in $\icat{WOp}\icat{N}(\cat{C}^f)$, its image in the category $\Fun(\Psi\opp,\icat{N}(\cat{C}^f))$ is equivalent to a strict functor $G:\Psi\opp\to\cat{C}^f$ since the middle horizontal map is an equivalence. But the fact that $F$ is in $\icat{WOp}\icat{N}(\cat{C}^f)$, it follows immediately that $G$ is in $\wop\cat{C}$. Hence the top horizontal map is essentially surjective.
\end{proof}

\begin{rem}
Instead of the category $\Psi$, we could have used the dendroidal category $\Omega$ as in \cite[Section 4]{boavidaoperads} in order to model homotopy coherent operads. 
\end{rem}

\section{Main theorem in the unstable case}

The pro-$p$ completion of spaces commutes with products. Indeed given two spaces $X$ and $Y$, there is an obvious map
\[\widehat{X\times Y}\to \widehat{X}\times\widehat{Y}\]
which is a weak equivalence as can be seen by computing the cohomology of both sides.

It follows that a levelwise application of pro-$p$-completion induces a functor
\[\h{(-)}:\icat{WOpS}\to \icat{WOp}\ips_p\]
that is the left adjoint of the functor
\[\Mat:\icat{WOp}\ips_p\to\icat{WOp}\is\]
obtained by objectwise application of the functor $\Mat$ (since $\Mat$ is a right adjoint it preserves products and hence weak operads).

We denote by $\od$ the operad in $\S$ given by applying the functor $\on{Sing}$ to the topological operad of little $2$-disks. The goal of this section is to prove the following theorem.

\begin{theo}\label{theo : main p}
The monoid of endomorphisms of $\h{N^\Psi\od}$ in the homotopy category of $\icat{WOp}\ips_p$ is $\pugt_p$.
\end{theo}

We follow the approach of \cite{horelprofinite}. Our starting point is the model category $\pS_p$ constructed by Morel and its groupoid companion $\pG_p$ constructed in Proposition \ref{prop: model structure on pro-p groupoids}.

We have the operad $\pab$ in the category of groupoids which is a groupoid model of $\od$ (see \cite[Section 6]{horelprofinite}). More precisely, there is a zig-zag of weak equivalences of operads in $\S$ connecting $B\pab$ and $\od$ (where $B$ denotes the nerve functor).

The levelwise pro-$p$ completion of $\pab$ is an operad $\h{\pab}_p$ in $\pG_p$. Our definition of the monoid $\pugt_p$ will be the monoid of endomorphisms of $\h{\pab}_p$ which induce the identity on objects. Note that this is not the standard definition, the monoid $\pugt_p$ is usually defined as the set of elements of $\ZZ_p\times \h{(F_2)}_p$ satisfying certain equations (see for instance \cite[Paragraph 4]{drinfeldquasi}). We refer the reader to \cite[Section 11.1.]{fressehomotopy} for a comparison of the two definitions. Note that Fresse focuses on the prounipotent case but the proof works exactly the same in the pro-$p$ case. We define the group $\pgt_p$ as the group of units of $\pugt_p$.

Let $I[1]$ be the groupoid completion of the category $[1]$. This is the groupoid that corepresents the functor that sends a groupoid to the sets of its morphisms. The endofunctor $C\mapsto C^{I[1]}$ in the category $\pG_p$ preserves products and hence induces an endofunctor $\oper{O}\mapsto\oper{O}^{I[1]}$ on the category of operads in $\pG_p$. Using this functor, we can define the notion of a homotopy between two maps in $\op\pG_p$. Namely, a homotopy between $f$ and $g$ two maps $\oper{P}\to\oper{Q}$ is a map $\oper{P}\to\oper{Q}^{I[1]}$ whose evaluation at the two objects of $I[1]$ is $f$ and $g$. We can form the category $\pi\op\pG_p$ whose objects are the operads in $\pG_p$ and morphisms are morphisms in $\op\pG_p$ modulo the homotopy relation.  

Our first step is the following proposition.

\begin{prop}
The map $\pugt_p\to\on{End}(\h{\pab})$ induces an isomorphism
\[\pugt_p\to\on{End}_{\pi\op\pG_p}(\h{\pab}).\]
\end{prop}

\begin{proof}
This is done exactly as \cite[Theorem 7.8.]{horelprofinite}.
\end{proof}

The levelwise application of profinite completion induces a weak equivalences preserving functor:
\[\h{(-)}:\wop\G\to\wop\pG_p.\]

\begin{prop}\label{prop-main theorem groupoids}
The action of $\pugt_p$ on $\h{\pab}_p$ induces an isomorphism
\[\pugt_p\cong\on{End}_{\on{Ho}\cat{WOp}\pG_p}(N^\Psi\h{\pab}).\]
\end{prop}

\begin{proof}
This can be proved as \cite[Proposition 8.1.]{horelprofinite}.
\end{proof}

Now, we are ready to prove Theorem \ref{theo : main p}.

\begin{proof}
By Proposition \ref{prop: weak operads infinity categories}, there is an isomorphism
\[\on{End}_{\on{Ho}\cat{WOp}\pS_p}(\h{N^{\Psi}B\pab})\cong\on{End}_{\on{Ho}\icat{WOp}\ips_p}(\h{N^{\Psi}B\pab}).\]
Thus, using the previous proposition, it is enough to prove that there is an isomorphism of monoids
\[\on{End}_{\on{Ho}\cat{WOp}\pG_p}(N^{\Psi}\h{\pab})\cong\on{End}_{\on{Ho}\cat{WOp}\pS_p}(\h{N^{\Psi}B\pab}).\]
This is done in the profinite case in \cite[Proposition 8.3.]{horelprofinite} using the goodness of the pure braid groups (\emph{cf.} \cite[Corollary 5.11.]{horelprofinite}). The fact that the pure braid groups are good as profinite groups implies immediately that they are $p$-good. Hence the proof of \cite[Proposition 8.3.]{horelprofinite} applies \emph{mutatis mutandi}. 
\end{proof}

Using Theorem \ref{theo : main p}, we can prove Theorem \ref{theo: main theorem hopf} and Theorem \ref{theo: main theorem BK}. For $\kk$ a commutative ring, we denote by $\icat{HCOp}_{\kk}$ the $\infty$-category $(\icat{WOp}((\icalg{\kk})\opp))\opp$ and call it the $\infty$-category of Hopf cooperads. It is very likely that if $\kk$ is a field of characteristic $0$, then the $\infty$-category of Hopf cooperads can be presented as the $\infty$-category underlying the model category constructed by Fresse in \cite{fressehomotopy}. The latter is a model structure on the category of strict cooperads in CDGAs over $\kk$ where the weak equivalences are the quasi-isomorphisms. However such a result should not be expected in characteristic $p$ which justifies the use of $\infty$-categories.

\begin{theo}
Let $\kk$ be an algebraically closed field of characteristic $p$. The monoid of endomorphisms of $\CC(N^\Psi\od,\kk)$ in the category $\on{Ho}\icat{HCOp}_{\kk}$ is $\pugt_p$.
\end{theo}

\begin{proof}
The functor $\CC(-,\kk):\ips_p\to \icalg{\kk}\opp$ preserves products and thus induces a functor
\[\CC(-,\kk):\icat{WOp}\ips_p\to \icat{HCOp}_{\kk}\opp.\]
From Proposition \ref{prop: cochains is fully faithful}, we deduce that there is a fully faithful embedding of $\infty$-categories
\[\CC(-,\kk):\icat{WOp}(\Pro(\is^0_{p-fin}))\to \icat{HCOp}_{\kk}\opp.\]
Moreover, the obvious map
\[\CC(\h{N^\Psi\od},\kk)\to\CC(N^\Psi\od,\kk)\]
is an equivalence of Hopf cooperads.
\end{proof}

Now, we come to the proof of Theorem \ref{theo: main theorem BK}. We start with the following theorem.

\begin{theo}\label{theo: partial result}
The materialization functor $\Mat:\ips_p\to \is$ induces an isomorphism
\[\End_{\on{Ho}\icat{WOp}\ips_p}(\h{N^\Psi\od})\cong \End_{\on{Ho}\icat{WOp}\is}(\Mat(\h{N^\Psi\od})).\]
\end{theo}

\begin{proof}
This can be proved exactly as \cite[Corollary 8.12]{horelprofinite} using the fact that the $p$-completion of the pure braid groups are strongly complete. This last fact follows from the main theorem of \cite{nikolovfinitely} together with the fact that those profinite groups are topologically finitely generated.
\end{proof}

Now we wish to compare the Bousfield-Kan construction and the functor $X\to \Mat(\h{X})$. We first construct a more explicit model of $X\mapsto \Mat(\h{X})$. Using the Quillen adjunction \ref{eq: quillen pair} and the sentence following it, we see that $X\mapsto \Mat(\h{X})$ can be modeled by $X\mapsto |(\h{X})^f|$ where the $f$-superscript denotes a functorial fibrant replacement. As explained in \cite{morelensembles}, any object $P$ in $\pS$ can be functorially expressed as $P=\on{lim}_{R\in \mathcal{R}(P)} P/R$ where $\mathcal{R}(P)$ is the filtered poset of continuous equivalence relations on $P$ that are such that $P/R$ is levelwise finite. 

For $X$ a simplicial set, let us denote by $\on{Res}^\bullet(X)$ the cosimplicial simplicial set whose totalization is $(\ZZ/p)_\infty(X)$. Morel shows in \cite[Proposition 2]{morelensembles}, that  for $P$ an object of $\pS$, the pro-object \[\{\on{Tot}^s(\on{Res}^\bullet( P/R))\}_{s\in\mathbb{N},R\in\mathcal{R}(P)}\] 
is a fibrant replacement of $P$ in $\pS_p$. It follows that the functor $X\mapsto \Mat(\h{X})$ is equivalent to the functor
\[M(X):= \on{lim}_{s\in\mathbb{N},R\in\mathcal{R}(\h{X})}\on{Tot}^s(\on{Res}^\bullet( \h{X}/R)).\]
There is also an obvious natural transformation
\[(\ZZ/p)_{\infty}(X):= \on{lim}_{s\in\mathbb{N}}\on{Tot}^s(\on{Res}^\bullet( X))\to M(X).\]

\begin{lemm}\label{lemm: BK vs pro} The above natural transformation is an equivalence on each simplicial set that are weakly equivalent to a levelwise finite simplicial set.
\end{lemm}

\begin{proof}
Since both functor are weak equivalence preserving, it suffices to prove this claim for $X$ a levelwise finite simplicial set. In that case $\h{X}=X$ and the poset of continuous equivalence relations with a levelwise finite quotient has an initial object. It follows that for such an $X$, the natural transformation $(\ZZ/p)_\infty(X)\to M(X)$ is actually an isomorphism.
\end{proof}

We can now prove Theorem \ref{theo: main theorem BK}.

\begin{theo}
The $\Psi$-diagram $(\ZZ/p)_\infty (N^\Psi\od)$ is an object of $\icat{WOp}\is$ and there is an isomorphism 
\[\pugt_p\cong\End_{\on{Ho}\icat{WOp}\is}((\ZZ/p)_\infty (N^\Psi\od)).\]
\end{theo}

\begin{proof}
The functor $(\ZZ/p)_\infty$ preserves products up to homotopy by \cite[I.7.2]{bousfieldhomotopy}, therefore $(\ZZ/p)_\infty (N^\Psi\od)$ is in $\icat{WOp}\is$. The spaces $\od(n)$ have the homotopy type of finite CW-complexes. Thus the previous lemma implies that $(\ZZ/p)_\infty (N^\Psi\od)\to  M(N^\Psi\od)$ is an equivalence in $\icat{WOp}\is$. But now, using Theorem \ref{theo: partial result}, we have an isomorphism
\[\pugt_p\cong\End_{\on{Ho}\icat{WOp}\is}(M(N^\Psi\od)).\]
\end{proof}

\begin{rem}
In \cite{boavidaoperads}, we proved that the group $\pgt$ is also the group of automorphisms of the profinite completion of the framed little disks operad. The argument above can be applied in order to prove that the group $\pgt_p$ is the group of automorphisms of the $p$-completed framed little disks operad. The framed little disks version of Theorem \ref{theo: main theorem hopf} and Theorem \ref{theo: main theorem BK} also hold.
\end{rem}

\section{Review of $p$-complete stable homotopy theory}

Given a stable $\infty$-category $\icat{C}$, we denote by $\map_{\icat{C}}(-,-)$ the mapping spectrum and by $\Map_{\icat{C}}(-,-)$ the mapping space (which is the infinite-loop space of the mapping spectrum).

For $E$ a spectrum, we denote by $L_pE$ the localization of $E$ with respect to the homology theory represented by $S\ZZ/p$ (the Moore spectrum of the group $\ZZ/p$). Note that by \cite[Theorem 3.1.]{bousfieldlocalization}, this coincides with the $H\ZZ/p$-localization if $E$ is bounded below. For $E$ a bounded below spectrum with finitely generated homotopy groups, the map $E\to L_pE$ induces the map $E_*\to E_*\otimes\ZZ_p$ on homotopy groups.

We denote by $\sp_{p-fin}$ the smallest full stable subcategory of $\sp$ containing the spectrum $H\ZZ/p$. This can also be described as the full subcategory of $\sp$ spanned by the spectra whose homotopy groups are finite $p$-groups and are almost all $0$.

We denote by $\psp_p$ the $\infty$-category $\Pro(\sp_{p-fin})$. The inclusion $\sp_{p-fin}\to \sp$ induces a limit preserving functor $\Mat:\psp_p\to\sp$. This has a left adjoint denoted $X\mapsto \h{X}$. 

We denote by $\tau_n$ the $n$-th Postnikov section endofunctor on $\sp$. By the universal property of the pro-category, there is a unique endofunctor of $\psp_p$ that coincides with $\tau_n$ on $\sp_{p-fin}$ and commutes with cofiltered limits. We still denote this functor $\tau_n$. For $A$ a pro-$p$ abelian group, we denote by $\h{H}A$ the object of $\psp_p$ given by applying the Eilenberg-MacLane functor to an inverse system of finite abelian group whose limit is $A$. Note for instance that $\h{H}\ZZ_p$ lives in $\psp_p$ while $H\ZZ_p$ lives in $\sp$. There is an obvious weak equivalence $H\ZZ_p\to \Mat(\h{H}\ZZ_p)$.

We denote by $\sp_p$ the full $\infty$-subcategory of $\sp$ spanned by spectra that are local with respect to $S\ZZ/p$. We denote by $\sp_p^{ft}$ the full $\infty$-subcategory of $\sp$ spanned by bounded below spectra whose homotopy groups are finitely generated $\ZZ_p$-modules. Note that $\sp_p^{ft}$ is a full $\infty$-subcategory of $\sp_p$. Note also that if $X$ is a bounded below spectrum that has finitely generated homotopy groups, then $L_pX$ is in $\sp_p^{ft}$. Similarly, we denote by $\psp_p^{ft}$ the full $\infty$-subcategory of $\psp_p$ spanned by pro-spectra that are bounded below and whose homotopy groups are finitely generated $\ZZ_p$-modules (given a pro-spectrum $X=\lim_{i\in I}X_i$ with $I$ a cofiltered category, its $n$-th homotopy group is, by definition, the pro-abelian group $\{\pi_n(X_i)\}_{i\in I}$).

\begin{lemm}\label{lemm: HZZ_p is good}
The map $H\ZZ_p\to\Mat(\h{H}\ZZ_p)$ is adjoint to a weak equivalence $\h{H\ZZ_p}\to \h{H}\ZZ_p$.
\end{lemm} 

\begin{proof}
It suffices to show that for any spectrum $F$ in $\sp_{p-fin}$, the map
\[\map_{\psp_p}(\h{H}\ZZ_p,F)\to\map(H\ZZ_p,F)\]
is a weak equivalence. Since both sides of the equation are exact in $F$, it suffices to do it for $F=H\ZZ/p$. Hence we are reduced to proving that the map
\[\on{colim}_n H^k(H\ZZ/p^n,\ZZ/p)\to H^k(H\ZZ_p,\ZZ/p)\]
is an isomorphism for each $k$. Since $\ZZ/p$ is a field, cohomology is dual to homology and it suffices to prove that $H_k(H\ZZ_p,\ZZ_p)$ is isomorphic to $\{H_k(H\ZZ/p^n,\ZZ/p)\}_n$ in the category of pro-abelian groups. In \cite[Proposition 3.3.10.]{luriedag13}, Lurie shows that there is an isomorphism of pro-abelian groups:
\[H_k(\Sigma^{-m}\Sigma^{\infty}K(\ZZ_p,m),\ZZ/p)\cong \{H_k(\Sigma^{-m}\Sigma^{\infty}K(\ZZ/p^n,m),\ZZ/p)\}_n.\]
By Freudenthal suspension theorem, for any abelian group $A$, the map 
\[\Sigma^{-m}\Sigma^{\infty}K(A,m)\to HA\]
is roughly $m$-connected. Thus, taking $m$ large enough, Lurie's result gives what we need.
\end{proof}

We denote by $\tau^n$ the fiber of the map $\id\to\tau_n$, for $\tau_n:\psp_p\to\psp_p$ and for $\tau_n:\sp\to\sp$. That is, there are cofiber sequences
\[\tau^nX\to X\to \tau_nX.\]
for all $X$ that are moreover functorial in $X$ (where $X$ lives either in $\sp$ or in $\psp_p$).

\begin{lemm}
Let $X$ be an object of $\psp_p$. The following conditions on $X$ are equivalent.
\begin{enumerate}
\item The map $\tau^nX\to X$ is an equivalence.
\item The map $\tau_nX\to \ast$ is an equivalence.
\item For all $k\leq n$, the group $H^k(X,\ZZ/p)$ is zero.
\item If $Y$ is such that $Y\to \tau_nY$ is an equivalence, then the space $\Map(X,Y)$ is trivial.
\end{enumerate}
If $X$ satisfies these equivalent conditions, we say that $X$ is $n$-connected.
\end{lemm}

\begin{proof}
The equivalence of (1) and (2) is obvious. The implication $(1)\implies (3)$ is obvious. 

We now prove the equivalence of (2) and (4). For this we consider the category $(\psp_p)_{\leq n}$ which is the pro-category of the $\infty$-category of spectra $Y$ with the following properties:
\begin{itemize}
\item The homotopy groups of $Y$ are almost all zero.
\item The non-zero homotopy groups of $Y$ are finite $p$-groups.
\item The homotopy groups of degree larger than $n$ are all $0$.
\end{itemize} 
This is canonically identified with the full subcategory of $\psp_p$ spanned by those $Y$ such that $Y\to\tau_nY$ is an equivalence. Now, one easily verifies that the inclusion $(\psp_p)_{\leq n}$ has a left adjoint given by $X\mapsto \tau_nX$. The desired equivalence follows from this adjunction.

We prove $(3)\implies (4)$. If $X$ satisfies $(3)$, we easily deduce by a Postnikov type argument that condition $(4)$ is satisfied for all spectra $Y$ that are bounded below, have $p$-finite homotopy groups and satisfy $Y\to\tau_nY$ is an equivalence. Now, let $Y$ be an object of $\psp_p$ that is such that $Y\to\tau_nY$ is an equivalence. Then $Y$ can be written as a cofiltered limit of spectra $Y_i$ that are bounded below, have $p$-finite homotopy groups and satisfy $Y\to\tau_nY$ is an equivalence. The assertion (4) follows from the fact that the functor $\Map(X,-)$ commutes with limits.
\end{proof}

\begin{lemm}\label{lemm: Mat preserves connectivity}
Let $Y$ be an $(-1)$-connected spectrum in $\psp_p^{ft}$, then $\Mat(Y)$ is $(-1)$-connected.
\end{lemm}

\begin{proof}
The claim is true if $Y$ is also $m$-truncated for some $m$. Indeed, in that case a Postnikov tower argument allows us to reduce to the cases of $Y=\widehat{H}\ZZ_p$ and $Y=H\ZZ/p^k$ in which case the result is trivial. In general, we may write $Y$ as the limit of its Postnikov tower $Y\simeq\on{lim}_m\tau_mY$. Then, we have $\Mat(Y)\simeq\on{lim}_m\Mat(\tau_m Y)$. By Milnor's short exact sequence, we see that $\pi_n(\Mat(Y))\cong \on{lim}_m\pi_n(\Mat(\tau_mY))$. In particular, $\Mat(Y)$ is $(-1)$-connected.
\end{proof}

\begin{lemm}\label{lemm: completion commutes with Postnikov}
Let $X$ be any spectrum, then the obvious map $\h{X}\to \lim_n\h{\tau_nX}$ is a weak equivalence.
\end{lemm}

\begin{proof}
Since the completion functor is exact, there are fiber sequences
\[\h{\tau^n X}\to \h{X}\to \h{\tau_nX}\]
that are functorial in $n$. Hence, it suffices to prove that $\lim_n\h{\tau^n X}$ is contractible. As in Lemma \ref{lemm: HZZ_p is good}, it is enough to prove that for any $k$, the group $H^k(\tau^n X,\ZZ/p)$ eventually becomes trivial for $n$ large enough. But this follows immediately from the fact that $\tau^nX$ is $n$-connected.
\end{proof}

\begin{prop}\label{prop: unit equivalence}
Let $Y$ be an object of $\sp_p^{ft}$. Then the unit map $Y\to\Mat(\h{Y})$ is a weak equivalence.
\end{prop}

\begin{proof}
Let us call a spectrum $Y$ good if this is the case. The good spectra form a triangulated subcategory of $\sp$. This subcategory contains $H\ZZ/p$. According to Lemma \ref{lemm: HZZ_p is good}, it also contains $H\ZZ_p$. Hence, it contains $\tau_nY$ for any $n$ and any $Y$ in $\sp_p^{ft}$. 

Thus, for $Y$ in $\sp_p^{ft}$, there is an equivalence $\tau_nY\to \Mat(\h{\tau_nY})$ for each $n$. In order to prove that $Y$ is good, it will be enough to prove that the map
\[\Mat(\h{Y})\to\on{lim}_n\Mat(\h{\tau_n Y})\]
is a weak equivalence. Since $\Mat$ is a right adjoint, it is enough to prove that the obvious map $\h{Y}\to\on{lim}_n\h{\tau_n Y}$ is a weak equivalence but this is the content of Lemma \ref{lemm: completion commutes with Postnikov}.
\end{proof}

The following corollary compares the two notions of $p$-completion for spectra.

\begin{coro}\label{coro: equivalence of the two completions}
There is a natural transformation from $L_p$ to $\Mat(\h{-})$ that is a weak equivalence when restricted to spectra $X$ such that $L_pX$ is in $\sp_p^{ft}$. In particular, it is a weak equivalence on spectra that are bounded below and have finitely generated homotopy groups.
\end{coro}

\begin{proof}
We first make the observation that for any spectrum $X$, the obvious map $\h{X}\to\h{L_pX}$ is a weak equivalence. Indeed, it suffices to prove that for any $F$ in $\sp_{p-fin}$, the map $X\to L_pX$ induces a weak equivalence
\[\map(L_pX,F)\to \map(X,F),\]
but this follows from the fact that $F$ is local with respect to $S\ZZ/p$. 

Thus, there is a natural transformation of endofunctors of $\sp$:
\[\alpha(X):L_pX\to \Mat(\h{L_pX})\simeq \Mat(\h{X})\]

Proposition \ref{prop: unit equivalence} tells us that $\alpha(X)$ is a weak equivalence whenever $L_p(X)$ is in $\sp_p^{ft}$ as desired. 
\end{proof}

\begin{prop}\label{prop: counit equivalence}
Let $Y$ be an object of $\psp_p^{ft}$, then the counit map $\h{\Mat(Y)}\to Y$ is an equivalence.
\end{prop}

\begin{proof}
This is obviously true for $Y$ in $\sp_{p-fin}$. This is true for $Y=\h{H}\ZZ_p$ by Lemma \ref{lemm: HZZ_p is good}. Hence, this is true for any pro-spectrum of the form $\h{H}A$ where $A$ is a finitely generated abelian pro-$p$ group. An easy Postnikov induction shows that this is true for any object $Y$ of $\psp_{p}^{ft}$ that is bounded above (i.e $Y\to \tau_n Y$ is an equivalence for some $n$).

Now a general $Y$ is the limit of the Postnikov tower: $Y\simeq \lim_n\tau_n Y$. The functor $\Mat$ is a right adjoint and hence preserves limits. Since $n$-truncated spectra are stable under limits, the map $\Mat(Y)\to\Mat(\tau_n Y)$ has an $n$-truncated target and an $n$-connected fiber by Lemma \ref{lemm: Mat preserves connectivity}, hence it exhibits $\Mat(\tau_n Y)$ as the $n$-truncation of $\Mat(Y)$. From these two facts, we see that are reduced to proving that the obvious map
\[(\lim_n\tau_n\Mat(Y))^{\wedge}\to Y\]
is  a weak equivalence. Finally, using Lemma \ref{lemm: completion commutes with Postnikov}, we can pull the Postnikov limit outside of the completion and we deduce that the counit is a weak equivalence as desired.
\end{proof}

We can now deduce our main result. 

\begin{theo}\label{theo: equivalence of categories}
The functors $\Mat$ and $\h{(-)}$ are mutually inverse equivalences between $\psp_p^{ft}$ and $\sp_p^{ft}$. 
\end{theo}

\begin{proof}
Using Proposition \ref{prop: counit equivalence} and Proposition \ref{prop: unit equivalence}, it is enough to prove that $\Mat$ sends $\sp_p^{ft}$ to $\psp_p^{ft}$ and that $\h{(-)}$ sends $\psp_p^{ft}$ to $\sp_p^{ft}$. 

Let $X$ be a bounded spectrum in $\psp_p^{ft}$, that is $X\to\tau_n X$ is an equivalence for $n$ large. Using Postnikov induction, we see that $X$ is in the triangulated subcategory of $\psp_p$ generated by $\h{H}\ZZ_p$. Since $\Mat(\h{H}\ZZ_p)\simeq H\ZZ_p$ and $\Mat$ is exact,we deduce that $\Mat(X)$ is in $\sp_p^{ft}$. Now if $X$ is in $\psp_p^{ft}$, we can write it as $\lim_n\tau_nX$. The functor $\Mat$ preserves limits and the Postnikov tower as was observed in the proof of Proposition \ref{prop: counit equivalence}. But a spectrum is in $\sp_p^{ft}$ if and only if all of its Postnikov sections are in $\sp_p^{ft}$ thus $\Mat(X)$ is in $\sp_p^{ft}$ as desired.

Now, we prove the converse direction. Let $Y$ be a spectrum in $\sp_p^{ft}$. According to Lemma \ref{lemm: completion commutes with Postnikov}, in order to prove that $\h{Y}$ is in $\psp_p^{ft}$, it suffices to prove that $\h{\tau_n Y}$ is in $\psp_p^{ft}$ for all $n$. Hence we may assume that $Y$ is bounded (i.e. $Y\to\tau_n Y$ is an equivalence). Again, by a Postnikov induction, we deduce that it suffices to prove that $\h{H\ZZ_p}$ is in $\psp_p^{ft}$. This follows from Lemma \ref{lemm: HZZ_p is good}.
\end{proof}

\section{Symmetric monoidal structure}

We recall the theory of symmetric monoidal localizations of $\infty$-categories. Given a symmetric monoidal $\infty$-category $\icat{C}$, a collection $S$ of morphisms in $\icat{C}$ is said to be compatible with the symmetric monoidal structure if for any object $c\in \icat{C}$, the functor 
\[c\otimes-:\icat{C}\to\icat{C}\]
sends the maps of $S$ to maps of $S$. In this situation, it can be shown that the localization $\icat{C}[S^{-1}]$ can be given an essentially unique symmetric monoidal structure that makes the localization functor $\icat{C}\to\icat{C}[S^{-1}]$ into a symmetric monoidal functor. This fact is proved in \cite[Section 3]{hinichdwyer}.

From this general theory, we deduce that the $\infty$-category $\sp_p$ has a symmetric monoidal structure denoted $\wedge_p$. Given two spectra $X$ and $Y$ in $\sp_p$, the spectrum $X\wedge_pY$ has the homotopy type of $L_p(X\wedge Y)$. 

Similarly, the category $\psp_p$ has a symmetric monoidal structure simply denoted $\wedge$. We construct it by first constructing a symmetric monoidal structure on $\Pro(\sp)$ of all pro-spectra. This symmetric monoidal structure is the unique one such that $\wedge$ preserves cofiltered limits separately in each variable and that coincides with the standard monoidal structure on $\sp$. Then we use the following proposition that shows that $\psp_p$ can be seen as the localization of $\Pro(\sp)$ at a class of maps that is compatible with the symmetric monoidal structure.

\begin{prop}
The obvious map $\psp_p\to\Pro(\sp)$ exhibits $\psp_p$ as the localization of $\Pro(\sp)$ with respect to the $H\ZZ/p$-cohomology equivalences.
\end{prop}

\begin{proof}
The $\infty$-category $\Pro(\sp)$ can be identified with the opposite of the $\infty$-category of exact functors from spectra to spectra. With this description, the left adjoint $\Pro(\sp)\to \psp_p$ takes a functor to its restriction to $\sp_{p-fin}$. It follows immediately that if a map is sent to an equivalence by this left adjoint, it must be an $H\ZZ/p$-cohomology equivalence. The converse follows from the observation that $\sp_{p-fin}$ can be described as the stable $\infty$-subcategory of $\sp$ generated by $H\ZZ/p$. Thus, if $f:X\to Y$ is an $H\ZZ/p$-cohomology equivalence, the map $\map(Y,F)\to\map(X,F)$ is an equivalence for any $F\in\sp_{p-fin}$.
\end{proof}
 
From this construction of the symmetric monoidal structure on $\psp_p$, we deduce the following proposition.

\begin{prop}
The completion functor $\h{-}:\sp\to\psp_p$ can be promoted to a symmetric monoidal functor.
\end{prop}

\begin{proof}
We can write it as the following composite
\[\sp\to\Pro(\sp)\to\psp_p,\]
where the first map and second map are symmetric monoidal.
\end{proof}

The functor $\sp\to\psp_p$ inverts the $p$-local equivalences. It follows that it factors through $\sp_p$ and by the universal property of the symmetric monoidal localization, we get the following proposition.

\begin{prop}
The functor $\h{(-)}:\sp_p\to\psp_p$ can be promoted to a symmetric monoidal functor.
\end{prop}

We can now state and prove the main result of this section.

\begin{prop}
The functors $\Mat$ and $\h{(-)}$ can be promoted to mutually inverse symmetric monoidal equivalences between $\sp_p^{ft}$ and $\psp_p^{ft}$.
\end{prop}

\begin{proof}
In light of the above proposition and of Theorem \ref{theo: equivalence of categories}, it suffices to prove that $\sp_p^{ft}$ is stable under the symmetric monoidal structure of $\sp_p$. That is, we want to show that, given two spectra $X$ and $Y$ in $\sp_p^{ft}$, the smash product $X\wedge_p Y$ is in $\sp_p^{ft}$. Using the cofiber sequences
\[\tau^nX\wedge_pY\to X\wedge_pY\to \tau_nX\wedge_pY\]
and the fact that $\ZZ/p$-localization preserves connectivity (see \cite[Proposition 2.5.]{bousfieldlocalization})
we see that for each integer $k$, the homotopy group $\pi_k(X\wedge_p Y)$ coincides with $\pi_k(\tau_nX\wedge_p Y)$ for $n$ large enough. Hence we can assume without loss of generality that $X$ is truncated. Using again an induction on the number of non-zero homotopy groups of $X$, we can reduce to the case where $X=HA$ with $A$ a finitely generated $\ZZ_p$-module. We can then form a similar reduction on the $Y$ variable and reduce ourselves to proving that, if $A$ and $B$ are two finitely generated $\ZZ_p$-modules, the group $\pi_0(L_p(HA\wedge HB))$ is a finitely generated $\ZZ_p$-module. We claim that in general, we have the formula
\[\pi_0(L_p(HA\wedge HB))\cong A\otimes_{\ZZ_p} B.\]
In order to prove this, we may reduce to the case when $B=\ZZ_p$ or $B=\ZZ/p^k$ for some $k$. We do the case $B=\ZZ_p$, the other case being similar but easier. By the fact that the localization $L_p$ is symmetric monoidal, we deduce that there is an equivalence
\[L_p(HA\wedge H\ZZ)\simeq L_p(HA\wedge H\ZZ_p).\]
Moreover, since $L_p$ preserves connectivity, we have 
\[A=\pi_0(L_p(HA\wedge\SS))\cong \pi_0(L_p(HA\wedge H\ZZ))\cong \pi_0(L_p(HA\wedge H\ZZ_p))\]
as desired.
\end{proof}

\section{Stable automorphism}

\subsection{Weak operads vs operads}

We denote by $\cat{OP}$ the prop that defines the structure of an operad. That is, $\cat{OP}$ is a symmetric monoidal category and if $(\cat{C},\otimes)$ is another symmetric monoidal category, there is an equivalence of categories between the category of operads in $\cat{C}$ and the category of symmetric monoidal functors $\cat{OP}\to\cat{C}$. We denote by $\icat{N}{\cat{OP}}$ the underlying symmetric monoidal $\infty$-category.

\begin{defi}
Let $\icat{C}$ be a symmetric monoidal $\infty$-category. We define the $\infty$-category $\iop\icat{C}$ of operads in $\icat{C}$ to be the $\infty$-category of symmetric monoidal functors $\icat{N}\cat{OP}\to\icat{C}$
\end{defi}

\begin{cons}
We construct a symmetric monoidal functor $\cat{OP}\to\Psi\opp$ where $\Psi\opp$ is given its cartesian symmetric monoidal structure (see Section \ref{section : weak operads} for the definition of $\Psi$). Consider the collection of representable functors on the category of operads in sets: $\cat{Op}\cat{Set}(F(n),-)$ where $F(n)$ denotes the free operad on an operation of arity $n$. This collection has the structure of an operad in the category of functors, which by Yoneda's lemma gives the collection of objects $F(n)$ the structure of an operad in the category $\Psi\opp$. This operad is classified by a symmetric monoidal map $u:\cat{OP}\to\Psi\opp$.
\end{cons}

\begin{cons}
We construct, for any $\infty$-category $\icat{C}$ with products, a functor
\[u^*:\icat{WOp}\icat{C}\to\iop\icat{C},\]
where $\iop\icat{C}$ is the $\infty$-category of symmetric monoidal functors \[\icat{N}\cat{OP}\to(\icat{C},\times)\]
The functor $u^*$ is obtained via the following zig-zag
\[\icat{WOp}\icat{C}\xleftarrow{\simeq}\Fun^{\otimes}(\icat{N}((\Psi\opp)),\icat{C})\to\iop\icat{C},\]
where the left hand side equivalence is given by \cite[Proposition 2.4.1.7]{luriehigheralgebra} and the right hand side map is induced by precomposition with the functor $u$ constructed in the previous paragraph.
\end{cons}

We have the $\infty$-category $\iop\sp$ of operads in spectra. We can also form the $\infty$-category $\iop\sp_p$ of operads in $\sp_p$. The localization functor $L_p$ induces a functor
\[\iop\sp\to\iop\sp_p.\]

\subsection{Unipotent completion}

We refer the reader to appendix A of \cite{hainweighted} for background material on the theory of unipotent completion. If $G$ is a discrete group and $\kk$ is a field of characteristic $0$, there exist a prounipotent group $G_{\kk}$ (denoted $G^{un}_{/\kk}$ in \cite{hainweighted}) which is universal with respect to maps $G\to U(\kk)$ with $U$ a unipotent algebraic group over $\kk$. Note that, by \cite[Proposition A.1]{hainweighted},  the group $G_{\kk}$ is just $(G_{\QQ})\times_{\QQ}\kk$.  Now if $G$ and $\kk$ have topologies, one can also form the continuous pro-unipotent completion of $G$ over $\kk$. This is a prounipotent group over $\kk$ denoted $G_{\kk}$ and which is universal with respect to continuous maps $G\to U(\kk)$ for $U$ a unipotent group scheme over $\kk$.

The important result for us will be the following:

\begin{theo}\label{theo: unipotent vs pro p}
Let $G$ be a finitely generated discrete group and $p$ be a prime number, then there is a natural isomorphism of prounipotent groups over $\QQ_p$:
\[G_{\QQ_p}\cong (\h{G}_p)_{\QQ_p}.\]
\end{theo}

\begin{proof}
This is \cite[Theorem A.6]{hainweighted}.
\end{proof}

The important takeaway of this theorem is  that the $\QQ_p$ unipotent completion of $G$ can be constructed functorially from the pro-$p$ completion of $G$.

If $C$ is a groupoid with finitely many object and $\kk$ is a field, we introduce the $\kk$-unipotent completion $C_\kk$ of $C$ which is universal with respect to maps from $C$ to unipotent groupoids over $\kk$ inducing the identity on objects. (A unipotent groupoid is a groupoid enriched in schemes over $\kk$ with the property that the automorphisms of any objects is a unipotent group.) In practice, if $C$ is connected, then $C$ is abstractly isomorphic to $G\times \on{Codisc}(S)$ (where $\on{Codisc}(S)$ denotes the groupoid whose set of objects is $S$ and in which any two objects are uniquely isomorphic). In that case $C_{\kk}$ is isomorphic to $G_{\kk}\times\on{Codisc}(S)$.

\subsection{The pro-algebraic Grothendieck-Teichmüller group}

Recall the operads $\pab$ of parentesized braids. Since all the objects sets are finite, we can form its pro-unipotent completion $\pab_{\QQ}$ over $\QQ$. For $R$ a $\QQ$-algebra, we define a group $\on{GT}(R)$ via the formula
\[\on{GT}(R)=\Aut_0((\pab_\QQ)\times_{\QQ} R),\]
where $\Aut_0$ denotes the group of automorphisms of operads (in the category of pro-algebraic groupoids over $R$) inducing the identity on objects. It is proved by Drinfel'd in \cite{drinfeldquasi} that $\on{GT}$ is an affine group scheme over $\QQ$ that sits in a short exact sequence
\[1\to \on{GT}_1\to \on{GT}\to \mathbb{G}_m\to 1\]
where $\on{GT}_1$ is prounipotent. By work of Brown, the group $\on{GT}$ receives an injective map from the Tannakian fundamental group of the category of $\cat{MTM}(\ZZ)$ of mixed Tate motives over $\ZZ$ (see \cite{brownmixed})

By Theorem \ref{theo: unipotent vs pro p} (or rather the obvious generalization for groupoids), we have an isomorphism of operads $\pab_{\QQ_p}\cong (\h{\pab}_p)_{\QQ_p}$. Therefore the action of $\pgt_p$ on $\h{\pab}_p$ induces a continuous map $\pgt_p\to\on{GT}(\QQ_p)$. This map is an injection (see \cite{drinfeldquasi}). It is also compatible with the cyclotomic character in the sense that it fits in a commutative square
\[\xymatrix{
\pgt_p\ar[d]\ar[r]^{\chi}& \ZZ_p^{\times}\ar[d]\\
\on{GT}(\QQ_p)\ar[r]& \QQ_p^{\times}
}
\]

\begin{theo}[Tamarkin]\label{theo: Tamarkin}
Let $\kk$ be a field of characteristic $0$. There is an injective map
\[\on{GT}(\kk)\to \Aut_{\on{Ho}\iop\imod_{H\kk}}(H\kk\wedge\Sigma^\infty_+\od).\]
\end{theo}

\begin{proof}
Let $\pab^{>0}$ be the operad obtained from $\pab$ by replacing the arity $0$ operations by the empty groupoid. Let $\pab^{>0}_\kk$ be its pro-unipotent completion at the field $\kk$. It comes with an action of the group $\on{GT}(\kk)$ by \cite{drinfeldquasi}. As explained in \cite{barnatanassociators}, a choice of a Drinfeld's associator induces an isomorphism $\pab^{>0}_\kk\cong\pacd^{>0}_\kk$ where $\pacd_\kk$ is the operad of chord diagrams. This choice of isomorphism induces a commutative square
\[\xymatrix{
\on{GT}_1(\kk)\ar[r]\ar[d]_{\cong}& \Aut(\pab_\kk^{>0})\ar[d]^{\cong}\\
\on{GRT}_1(\kk)\ar[r]&\Aut(\pacd^{>0}_\kk)
}
\]
The bottom map of this commutative diagram restricts to a morphism
\[\on{GRT}_1(\kk)\to \Aut^0_{\on{Ho}\iop\icat{Ch}_\kk}(\C(\pacd^{>0}_\kk,\kk)).\]
where $\Aut^0$ denotes the group of automorphisms that induce the identity on homology. Both the source and the target of this map are pro-unipotent groups. Hence, this map is an injection if and only if the induced map on the Lie algebras is an injection. This last fact is exactly the main theorem of \cite{tamarkinaction}, that is the map
\[\mathfrak{grt}_1\to H^0(\on{Der}(\C(\pacd_\kk,\kk))\]
is an injection of pro-nilpotent Lie algebras. From this, we deduce that the map
\[\on{GT}_1(\kk)\to \Aut^0_{\on{Ho}\iop\icat{Ch}_\kk}(\C(\pab^{>0}_\kk,\kk))\]
is injective. On the other hand, the action of $\on{GT}(\kk)$ on $\pab_{\kk}^{>0}$ extends to an action of $\pab_{\kk}$. The induced map
\[\on{GT}_1(\kk)\to \Aut^0_{\on{Ho}\iop\icat{Ch}_\kk}(\C(\pab_\kk,\kk))\]
factors the previous map and hence must also be injective. This map fits in the following commutative diagram of short exact sequences
\[\xymatrix{
1\ar[r]& \on{GT}_1(\kk)\ar[r]\ar[d]&\on{GT}(\kk)\ar[r]\ar[d]&\kk^{\times}\ar[r]\ar[d]&1\\
1\ar[r]&\Aut^0(\C(\pab_\kk))\ar[r]&\Aut(\C(\pab_{\kk}))\ar[r]&\Aut(\on{H}_*(\pab_\kk))\ar[r]&1
}
\]
The right vertical map in this diagram is an isomorphism. It follows that the middle vertical map is injective. Finally the equivalence of symmetric monoidal $\infty$-categories between $\icat{Ch}_{\kk}$ and $\imod_{H\kk}$ induce an isomorphism
\[\Aut_{\on{Ho}\iop\icat{Ch}_{\kk}}(\C(\od,\kk))\cong \Aut_{\on{Ho}\iop\imod_{H\kk}}(H\kk\wedge\Sigma^\infty_+\od).\]
\end{proof}

\begin{rem}
In \cite{willwacherkontsevich}, Willwacher promotes this result to an isomorphism 
\[\on{GT}(\kk)\cong \on{Aut}_{\on{Ho}\iop\icat{Ch}_\kk}(C_*(\od^{>0},\kk)).\]
We will however not need this stronger result.
\end{rem}

\subsection{Proof of Theorem \ref{theo: main theorem stable}}

Our goal is now to prove a Theorem \ref{theo: main theorem stable} which is a $p$-complete version of Tamarkin's theorem.

For $X$ a spectrum, we define $H\ZZ_p\barwedge X$ by the following formula
\[H\ZZ_p\barwedge X :=\on{lim}_n H\ZZ/p^n\wedge X,\]
where the limit is taken in the $\infty$-category of $H\ZZ_p$-modules. This construction defines a functor from spectra to $H\ZZ_p$-modules. Note that there is a natural transformation
\[H\ZZ_p\wedge X\to H\ZZ_p\barwedge X.\]
We will need the following two lemmas about this construction.

\begin{lemm}\label{lemm: completed tensor product}
For any bounded below spectrum $X$, the map $X\to L_{p}X$ induces a weak equivalence
\[H\ZZ_p\barwedge X\to H\ZZ_p\barwedge L_{p}X.\]
\end{lemm}

\begin{proof}
First, we observe that all the spectra $H\ZZ/p^n$ are in the same Bousfield class (i.e. define the same localization functor). It suffices to prove that they have the same acyclics. The formula
\[H\ZZ/p\wedge X\simeq H\ZZ/p\wedge_{H\ZZ/p^n} H\ZZ/p^n\wedge X\]
shows that $H\ZZ/p^n$-acyclics are $H\ZZ/p$-acyclics. Conversely, the Bockstein cofiber sequences
\[H\ZZ/p\to H\ZZ/p^{n+1}\to H\ZZ/p^n\]
inductively implies that  $H\ZZ/p$-acyclics are $H\ZZ/p^n$-acyclics for all $n$. 

From this observation and the fact that for bounded below spectra, $L_p$ coincides with $H\ZZ/p$-localization, we deduce that the map $X\to L_{p}X$ induces an equivalence of $H\ZZ_p$-modules $H\ZZ/p^n\wedge X\to H\ZZ/p^n\wedge L_{p}X$ for each $n$. Taking the homotopy limit with respect to $n$ gives the desired result.
\end{proof}

\begin{lemm}\label{lemm: equivalence on finite spectra}
The natural transformation 
\[H\ZZ_p\wedge X\to H\ZZ_p\barwedge X\]
is a weak equivalence on finite spectra.
\end{lemm}

\begin{proof}
This map is obviously a weak equivalence on the sphere spectrum. Moreover both functors are exact. Therefore the map must be an equivalence on any finite spectrum.
\end{proof}

Now, we extend the functor $H\ZZ_p\barwedge-$ to operads. This is easy to do since limits in the $\infty$-category of operads in $H\ZZ_p$-modules are computed aritywise. Thus for $\oper{O}$ an operad in spectra, we may define
\[H\ZZ_p\barwedge \oper{O}:=\on{lim}_nH\ZZ/p^n\wedge \oper{O},\]
where the limit is taken in the $\infty$-category $\iop\imod_{H\ZZ_p}$ of operads in $H\ZZ_p$-modules. We moreover have an equivalence 
\[(H\ZZ_p\barwedge\oper{O})(n)\simeq H\ZZ_p\barwedge\oper{O}(n).\]

We may now state the main theorem of this section.

\begin{theo}
There is a faithful action of $\pgt_p$ on $L_{p}\Sigma^{\infty}_+\od$ in the homotopy category of $\iop\sp_p$.
\end{theo}

\begin{proof}
The infinite loop space functor $\Omega^\infty$ is a finite limit preserving functor 
\[\Omega^\infty:\sp_{p-fin}\to\is_{p-fin},\]
therefore by the universal property of pro-categories it extends uniquely to a limit preserving functor $\Omega^{\infty}:\psp_p\to \ips_p$. We denote its left adjoint by $\Sigma^{\infty}_+$. If we identify the $\infty$-category $\psp_p$ with the opposite of the $\infty$-category of exact functors $\sp_{p-fin}\to \is$, the functor $\Sigma^{\infty}_+$ can be informally described by the formula
\[X=\{X_i\}_{i\in I}\in\Pro(\is_{p-fin})\mapsto\on{colim}_i\Map_{\sp}(\Sigma^{\infty}_+X_i,-).\]
Now, we can consider the following diagram of left adjoint functors
\[\xymatrix{
\is\ar[r]^{\h{(-)}}\ar[d]_{\Sigma^\infty_+}&\ips_p\ar[d]^{\Sigma^\infty_+}\\
\sp\ar[r]_{\h{(-)}}&\psp_p
}
\]
It commutes because the corresponding diagram of right adjoint commutes. From this fact and Corollary \ref{coro: equivalence of the two completions}, we deduce that if $X$ is a space which is such that $\Sigma^\infty_+(X)$ has finitely generated homotopy groups (e.g. if $X$ has the homotopy type of a finite CW-complex), then there is a weak equivalence
\[L_p(\Sigma^{\infty}_+X)\simeq \Mat(\Sigma^\infty_+(\h{X})).\]

Now, we claim that the functor $\Sigma^\infty_+:\ips_p\to\psp_p$ can be given the structure of a symmetric monoidal functor. In order to do so, we observe that it can be written as the following composite of symmetric monoidal functors
\[\ips_p\to \Pro(\is)\xrightarrow{\Sigma^\infty_+}\Pro(\sp)\xrightarrow{L}\psp_p,\]
where the first map is just the inclusion and the map $L$ is the localization with respect to the $H\ZZ/p$-cohomology equivalences.

From the above two observations, we deduce that $\Mat(\Sigma^\infty_+\h{\od})$ has an operad structure and is weakly equivalent to $L_p(\Sigma^\infty_+\od)$. But we also know that $\Mat(\Sigma^\infty_+\h{\od})$ is equipped with an action of the group $\pgt_p$. Transferring this action along the weak equivalence we deduce that $L_{p}\Sigma^{\infty}_+\od$ also has an action of $\pgt_p$. 

Thus, applying the functor $H\ZZ_p\barwedge$ and using Lemma \ref{lemm: completed tensor product}, we find a weak equivalence of operads in $H\ZZ_p$-modules
\[H\ZZ_p\barwedge \Sigma^{\infty}_+\od\to H\ZZ_p\barwedge L_{p}\Sigma^\infty_+\od\]
where the target has an action of $\pgt_p$. Now, using Lemma \ref{lemm: equivalence on finite spectra}, we find that the map
\[H\ZZ_p\wedge\Sigma^{\infty}_+\od\to H\ZZ_p\barwedge\Sigma^{\infty}_+\od\]
is a weak equivalence of operads in $H\ZZ_p$-modules. Combining these two equivalences, we have an equivalence of operads in $H\ZZ_p$-modules.
\[H\ZZ_p\wedge\Sigma^{\infty}_+\od\to H\ZZ_p\barwedge L_{p}\Sigma^{\infty}_+\od\]
where the target comes with an action of $\pgt_p$. Now, we can apply the functor $H\QQ_p\wedge_{H\ZZ_p}$ to this weak equivalence and we find a weak equivalence of operads in $H\QQ_p$-modules:
\[H\QQ_p\wedge\Sigma^{\infty}_+\od\to H\QQ_p\wedge_{H\ZZ_p}H\ZZ_p\barwedge L_{p}\Sigma^{\infty}_+\od\]
where the target comes with an action of $\pgt_p$. Thus, taking the automorphisms in the homotopy category of operads in $H\QQ_p$-modules, we find a map
\[\pgt_p\to\Aut_{\on{Ho}\iop\imod_{H\QQ_p}}(H\QQ_p\wedge\Sigma^{\infty}_+\od)\]
that factors through $\on{GT}(\QQ_p)$ as follows:
\[\pgt_p\to \on{GT}(\QQ_p)\to  \Aut_{\on{Ho}\iop\imod_{H\QQ_p}}(H\QQ_p\wedge\Sigma^{\infty}_+\od)\]
where the first map is the inclusion and the second map is the map that appears in Theorem \ref{theo: Tamarkin}. From that theorem, we deduce that the action of $\pgt_p$ on $H\QQ_p\wedge\Sigma^{\infty}_+\od$ is faithful. Since this action comes from an action on $L_p\Sigma^\infty_+\od$, we deduce that the action of $\pgt_p$ on that operad is also faithful as desired.
\end{proof}

\begin{rem}
In light of the above Theorem it seems reasonable to expect that $\pgt_p$ is indeed the group of automorphisms of $L_p\Sigma^\infty_+\od$. At this stage, a stable homotopy theorist will be tempted to introduce the groups $\pgt^{(n)}_p$ of automorphisms of $L_n\Sigma^\infty_+\od$ where $L_n$ denotes the localization with respect to the height $n$ and prime $p$ Morava $E$-theory spectrum. These groups fit into a tower
\[\ldots\to \pgt_p^{(n)}\to \pgt_p^{(n-1)}\to\ldots\to \pgt_p^{(0).}\]
By the main theorem of \cite{willwacherkontsevich}, we have $\pgt_p^{(0)}\cong \gt(\QQ_p)$, on the other hand, the limit of this tower is the group of automorphisms of $L_p\Sigma^\infty_+\od$ which   is conjecturally the group $\pgt_p$. In any case, since the map $\pgt_p\to \pgt_p^{(0)}$ is injective, and factors through $\pgt_p^{(n)}$, we deduce that for each $n$, there is an inclusion $\pgt_p\subset\pgt_p^{(n)}$.
\end{rem}

We conclude this paper with the following Proposition that shows that in the stable case, the action of $\pgt_p$ on $L_p\Sigma^\infty_+\od$ quite trivial in each arity.

\begin{prop}\label{prop: levelwise triviality}
For each $n$, the action of $\pgt_p$ on $L_p\Sigma^\infty_+\od(n)$ factors through the cyclotomic character. 
\end{prop}

\begin{proof}
First, we observe that $\Sigma_+^\infty\od(n)$ has the homotopy type of a finite wedge of spheres. Moreover any element of $\pgt_p$ that is in the kernel of the cyclotomic character acts trivially on $H_*(\od(n),\ZZ_p)\cong\pi_*(H\ZZ_p\wedge_p L_p \Sigma_+^\infty\od(n))$. Indeed, since the action of $\pgt_p$ is compatible with the operad structure it suffices to check it in arity $2$ which is where the generators of $H_*(\od(n),\ZZ_p)$ live. Thus the result follows from the following lemma. 
\end{proof}

\begin{lemm}
Let $X$ be a finite wedge of $p$-completed spheres equipped with a self-map $g$. Assume that $g$ induces the identity on $\pi_*(X\wedge_p H\ZZ_p)$, then $g$ is homotopic to the identity.
\end{lemm}

\begin{proof}
Let us denote by $H_*(X)$ the homotopy groups of $X\wedge_pH\ZZ_p$. The Atiyah-Hirzebruch spectral sequence gives us an isomorphism
\[\pi_*(X)\cong H_*(X)\otimes_{\mathbb{Z}_p}\pi_*(\mathbb{S}_p)\]
that is natural with respect to endomorphisms of $X$. It thus follows from the hypothesis that $g$ acts by the identity on the homotopy groups of $X$. Now, we observe that if $X$ is a finite wedges of $p$-completed spheres and $Y$ is any $p$-complete spectrum, the map
\[\Hom_{\on{Ho}\sp_p}(X,Y)\to \Hom_{\mathrm{Mod}_{\pi_*(L_p\mathbb{S})}}(\pi_*(X),\pi_*(Y))\]
is an isomorphism. In particular, we deduce that the map $g$ must be homotopic to the identity.
\end{proof}

\bibliographystyle{alpha}

\bibliography{biblio}

\begin{thebibliography}{BdBHR19}

\bibitem[BdBHR19]{boavidaoperads}
P.~Boavida~de Brito, G.~Horel, and M.~Robertson.
\newblock Operads of genus zero curves and the {G}rothendieck--{T}eichm{ü}ller
  group.
\newblock {\em Geometry \& Topology}, 23(1):299--346, 2019.

\bibitem[Ber06]{bergnerrigidification}
J.~Bergner.
\newblock Rigidification of algebras over multi-sorted theories.
\newblock {\em Algebraic \& Geometric Topology}, 6(4):1925--1955, 2006.

\bibitem[BF04]{bergercombinatorial}
C.~Berger and B.~Fresse.
\newblock Combinatorial operad actions on cochains.
\newblock {\em Math. Proc. Cambridge Philos. Soc.}, 137(1):135--174, 2004.

\bibitem[BHH17]{barneapro}
I.~Barnea, Y.~Harpaz, and G.~Horel.
\newblock Pro-categories in homotopy theory.
\newblock {\em Algebraic \& Geometric Topology}, 17(1):567--643, 2017.

\bibitem[BK72]{bousfieldhomotopy}
A.~K. Bousfield and D.~M. Kan.
\newblock {\em Homotopy limits, completions and localizations}.
\newblock Lecture Notes in Mathematics, Vol. 304. Springer-Verlag, Berlin-New
  York, 1972.

\bibitem[BN98]{barnatanassociators}
D.~Bar-Natan.
\newblock On associators and the {G}rothendieck-{T}eichmuller group. {I}.
\newblock {\em Selecta Math. (N.S.)}, 4(2):183--212, 1998.

\bibitem[Bou79]{bousfieldlocalization}
A.~Bousfield.
\newblock The localization of spectra with respect to homology.
\newblock {\em Topology}, 18(4):257--281, 1979.

\bibitem[Bro12]{brownmixed}
F.~Brown.
\newblock Mixed {T}ate motives over {$\mathbb{Z}$}.
\newblock {\em Annals of Mathematics}, 175(2):949--976, 2012.

\bibitem[Dri90]{drinfeldquasi}
V.~Drinfeld.
\newblock On quasitriangular quasi-hopf algebras and on a group that is closely
  connected with {$\mathrm{Gal}(\bar{Q}/Q)$}.
\newblock {\em Algebra i Analiz}, 2(4):149--181, 1990.

\bibitem[Dwy74]{dwyerstrong}
W.~G. Dwyer.
\newblock Strong convergence of the {E}ilenberg-{M}oore spectral sequence.
\newblock {\em Topology}, 13:255--265, 1974.

\bibitem[Fre17]{fressehomotopy}
Benoit Fresse.
\newblock {\em Homotopy of Operads and Grothendieck-Teichmuller Groups}.
\newblock American Mathematical Soc., 2017.

\bibitem[Goe95]{goersssimplicial}
P.~Goerss.
\newblock Simplicial chains over a field and {$p$}-local homotopy theory.
\newblock {\em Math. Z.}, 220(4):523--544, 1995.

\bibitem[Hin15]{hinichrectification}
V.~Hinich.
\newblock Rectification of algebras and modules.
\newblock {\em Doc. Math.}, 20:879--926, 2015.

\bibitem[Hin16]{hinichdwyer}
V.~Hinich.
\newblock Dwyer–{K}an localization revisited.
\newblock {\em Homology, Homotopy and Applications}, 18(1):27--48, 2016.

\bibitem[HM03]{hainweighted}
R.~Hain and M.~Matsumoto.
\newblock Weighted completion of {G}alois groups and {G}alois actions on the
  fundamental group of {$\Bbb P^1-\{0,1,\infty\}$}.
\newblock {\em Compositio Math.}, 139(2):119--167, 2003.

\bibitem[Hor17]{horelprofinite}
G.~Horel.
\newblock Profinite completion of operads and the
  {G}rothendieck--{T}eichm{ü}ller group.
\newblock {\em Advances in Mathematics}, 321:326--390, 2017.

\bibitem[Lur17a]{luriedag13}
J.~Lurie.
\newblock Derived algebraic geometry {XIII}.
\newblock {\em preprint available on http://www.math.harvard.edu/~lurie/},
  2017.

\bibitem[Lur17b]{luriehigheralgebra}
J.~Lurie.
\newblock Higher algebra.
\newblock {\em book available on http://www.math.harvard.edu/~lurie/}, 2017.

\bibitem[Man01]{mandelleinfinity}
M.~Mandell.
\newblock {$E_\infty$} algebras and {$p$}-adic homotopy theory.
\newblock {\em Topology}, 40(1):43--94, 2001.

\bibitem[Mor96]{morelensembles}
F.~Morel.
\newblock Ensembles profinis simpliciaux et interpr\'etation g\'eom\'etrique du
  foncteur {$T$}.
\newblock {\em Bull. Soc. Math. France}, 124(2):347--373, 1996.

\bibitem[NS07]{nikolovfinitely}
N.~Nikolov and D.~Segal.
\newblock On finitely generated profinite groups, {I}: strong completeness and
  uniform bounds.
\newblock {\em Annals of mathematics}, 165(1):171--238, 2007.

\bibitem[Tam02]{tamarkinaction}
D.~Tamarkin.
\newblock Action of the {G}rothendieck-{T}eichmueller group on the operad of
  {G}erstenhaber algebras.
\newblock {\em arXiv preprint math/0202039}, 2002.

\bibitem[Wil15]{willwacherkontsevich}
T.~Willwacher.
\newblock M. {K}ontsevich's graph complex and the
  {G}rothendieck-{T}eichm\"uller {L}ie algebra.
\newblock {\em Invent. Math.}, 200(3):671--760, 2015.

\end{thebibliography}

\end{document}